\documentclass{amsart}

\usepackage[dvipdfmx]{graphicx}
\usepackage{amsmath}
\usepackage{amssymb}
\usepackage{amsthm}
\usepackage{color,xypic}
\usepackage{tabmacD}

\title[Tropical integrable systems and Young tableaux]{Tropical integrable systems and Young tableaux:\\ Shape equivalence and Littlewood-Richardson correspondence}
\author{Shinsuke Iwao}
\address{Department of Mathematics, Tokai University, 4-1-1, Kitakaname, Hiratsuka, Kanagawa 259-1292, Japan.}
\email{iwao@tokai.ac.jp}
\date{\today}

\newtheorem{thm}{Theorem}[section]
\newtheorem{prop}[thm]{Proposition}
\newtheorem{lemma}[thm]{Lemma}
\newtheorem{defi}[thm]{Definition}
\newtheorem{example}[thm]{Example}
\newtheorem{rem}[thm]{Remark}
\newtheorem{conj}[thm]{Conjecture}
\newtheorem{cor}[thm]{Corollary}

\def\RR{\mathord{\mathbb{R}}}
\def\ZZ{\mathord{\mathbb{Z}}}
\def\kasane#1#2{\genfrac{}{}{0pt}{2}{#1}{#2}}
\def\la{\leftarrow}

\newcommand{\Tableau}[1]{%
\def\emp{\emptyset\bl}
\def\maru##1{\scalebox{1.3}{\lower 7pt\hbox{\textcircled{\scriptsize$##1$}}}\bl}%
\def\tama##1{\scalebox{1.3}{\lower 7pt\hbox{\textcircled{\scriptsize$##1$}}}}%
{\text{\tableau[sY]{#1}}}%
}%
\newcommand{\bm}[1]{\mbox{\boldmath{$#1$}}}

\makeatletter

\def\h@iti#1{%
\bgroup%
\let\\=\cr%
\def\r{\hbox to 0pt{$\to$}}%
\def\d{\vbox to 0pt{\hbox to 0pt{$\searrow$}}}%
\vcenter{\tabskip=0pt\halign{&$##$\hspace{10pt}\cr#1\crcr}}%
\egroup%
}

\newcommand{\haiti}[1]{\hspace{4.5pt}\h@iti{#1}\hspace{-4.5pt}}

\makeatother

\def\sayou#1{%
\bgroup%
\let\\=\cr%
\def\+{\hbox{\lower 5pt\hbox{\scalebox{3.5}{$+$}}}}%
\def\-{\hbox{\lower 5pt\hbox{\scalebox{3.5}{$+$}\lower -2pt\hbox to 0pt{\hspace{-19.9pt}\scalebox{2.5}{$\circ$}}}}}%
\vcenter{\tabskip=0pt\halign{&\hfil$##$\hfil\cr#1\crcr}}%
\egroup%
}

\def\o#1{\overline{#1}}

\begin{document}

\maketitle

\begin{abstract}
We present a new characterization of the shape equivalent class and the Littlewood-Richardson correspondence of Young tableaux in terms of tropical (ultradiscrete) integrable systems.
As an application, an alternative proof of the ``shape change theorem'' is given.
\end{abstract}

\section{Introduction}

The {\it tropicalization} is the transformation
\[
a+b\mapsto \min[A,B],\qquad ab\mapsto A+B,\qquad a^{-1}\mapsto -A,
\]
through which the structure of rings $(+,\times,{}^{-1})$ is transformed into the structure of semi-fields $(\min,+,-)$.
For example, the tropicalization of a polynomial function is a piecewise linear function.
It is known there are a lot of interesting examples in the field of combinatorics, mathematical physics, {\it etc.}~where the tropicalization provides a new insight and an application.

One of the most significant case is the {\it Young tableau}.
The earliest study on the tropical aspects of Young tableaux was made by A.~N.~Kirillov~\cite{2001phco.conf...82K}, who introduced the {\it geometric RSK correspondence} (originally, tropical RSK correspondence\footnote{The word ``tropical'' nowadays has a different meaning.
Many researchers prefer to use the ``geometric RSK correspondence'' instead.}).
This correspondence was studied further by M.~Noumi and Y.~Yamada~\cite{noumi2004} by means of tropical (ultradiscrete) integrable systems.
In~\cite{noumi2004}, they showed the fact that the {\it row-bumming algorithm} is expressed as a recurrence equation of tropical matrices, which is the tropicalization of the {\it discrete Toda equation}.
This technique has been accepted as a fundamental tool for studies on the combinatorics of Young tableau and related topics. 
(For resent studies on the geometric RKS correspondence and its applications, see \cite{corwin2014tropical,nguyen2016variants,o2014geometric} and references therein.)

On the other hand, Y.~Mikami~\cite{mikami2012en}, Y.~Katayama, and S.~Kakei~\cite{kakei2015en} introduced another new correspondence between Young tableaux and tropical integrable systems.
Interestingly enough, their correspondence is apparently independent of Noumi and Yamada's correspondence.
In 2018, Iwao~\cite{iwao2018rims} presented a tropical characterization of the {\it rectification of skew tableaux} based on these correspondences.
We would expect that the combination of Noumi-Yamada's geometric tableaux and Katayama-Kakei's correspondence provides a rich tropical interpretations of the combinatorics of Young tableaux.

\renewcommand\thesubsubsection{\arabic{section}.\arabic{subsubsection}}

This work is a continuation of~\cite{iwao2018rims}.
In Section \ref{sec:2}, we briefly review the previous work~\cite{iwao2018rims}.
The jeu de taquin slide, which is a fundamental procedure of the combinatorics of Young tableaux, is expressed by the recurrence formula (\ref{eq:QandW}).
This formula is equivalent to the combinatorial procedure $\varphi_k$ (\S \ref{sec:2.3}).
The rectification of Young tableaux is now expressed as a composition of finitely many $\varphi_k$'s (see (\ref{eq:diagram-of-rectifiction})).
It is diagrammatically expressed by a planar diagram such as (\ref{eq:diagram-of-rectifiction}).
From this diagram, one can induce some {\it horizontal} diagram (see (\ref{eq:diagram-of-rectifiction-tate})) by getting each column ``together in one bundle.''
It is showed that these horizontal diagrams can be defined independently of the choice of planer diagrams.

In Section \ref{sec:3}, we deal with the ``dual'' object to the previous section.
It would be natural to ask ``what will happen if one gets each {\it row} of the planer diagram together in one bundle?''
The result is some {\it vertical} diagram (see (\ref{eq:diagram-of-rectification-yoko})).
We show (Proposition \ref{prop:map-of-UI}) that this vertical diagram is determined independently of the choice of planer diagrams.
The proof is based on the fact that the diagram is characterized by some $\mathcal{L}$-formula (\S \ref{sec:1.3}).
As an application, we give a proof of the ``shape change theorem~\cite{fulton_1996}'' (Theorem \ref{thm:main1}).

One can do both procedures for rows and columns simultaneously.
In Section \ref{sec:4}, we introduce some ``concentrated'' diagram by getting each row and each column together in one bundle (see (\ref{eq:diagram-of-rectification-compact})).
It is shown that this diagram is closely related with the {\it Littlewood-Richardson correspondence}.
Moreover, we present a new characterization of the Littlewood-Richardson correspondence in terms of the tropical mathematics (Theorem \ref{thm:main2}).

For convenience of readers, we gave a short introduction in a combinatorics of Young tableaux in \S \ref{sec:appA}.
The definition of the Takahashi-Satsuma Box-Ball system is given in \S \ref{sec:appB}.

\subsubsection{Notations for Young tableaux}

In this article, we follow the conventions in Fulton's book \cite{fulton_1996}.
Let $\lambda=(\lambda_1\geq \lambda_2\geq \dots\geq \lambda_\ell)$ be a Young diagram.
A {\it semi-standard tableau of shape $\lambda$} is obtained by filling the boxes in $\lambda$ with a number according to the following rules:
(i) in each row, the numbers are weakly increasing from left to right,
(ii) in each column, the numbers are strongly increasing from top to bottom.
A semi-standard tableau is often referred to as {\it tableau} shortly.
A tableau with $n$ boxes is called {\it standard} if it contains distinct $n$ numbers $1,2,\dots,n$.
Let $\lambda/\mu$ be a skew diagram.
A {\it skew (semi-standard) tableau of shape $\lambda/\mu$} is obtained by filling the boxes with a number according to the same rule for a tableau.
If a skew tableau with $n$ boxes contains distinct $n$ numbers $1,\dots,n$, it is said to be {\it standard}.
See \S\ref{sec:appA} for other definitions.

\subsubsection{Notations for tropicalization}\label{sec:1.3}

We use the notations of mathematical logic in order to simplify arguments for tropicalization.
For details, see \cite{iwao2018rims}.

Let $\mathcal{L}=\{+,\cdot,{}^{-1},1\}$ be a language, where $+$ and $\cdot$ are binary function symbols, ${}^{-1}$ is a unary function symbol, and $1$ is a constant symbol.
Define the $\mathcal{L}$-structures $\mathcal{M}=(M,+,\cdot,{}^{-1},1)$ and $\o{\mathcal{M}}=(\o{M},\o{+},\o{\cdot},\o{{}^{-1}},\o{1})$, and the morphism $M\to \o{M}; x\mapsto \overline{x}$ of $\mathcal{L}$-structures as follows:
\begin{itemize}
\item $M$ is the set of germs at $\epsilon=0$ of continuous functions $f(\epsilon)$ of $\epsilon>0$ that satisfy $\lim\limits_{\epsilon\to +0}\epsilon\log f(\epsilon)\in \RR$.
Here $+,\cdot,{}^{-1}$ are standard addition, multiplication, and multiplicative inverse\footnote{Abstractly, $M$ can be replaced with any semi-field $M^{\ast}=(M^\ast,+,\cdot,{}^{-1},1)$ such that (i) there exists a surjection $M^\ast\to \o{M}$ of $\mathcal{L}$-structures, and (ii) $\mathrm{Th}(M^\ast)\supset \mathrm{Th}(M)$.
Here $\mathrm{Th}(M)$ is the {\it theory} of $M$, which is the set of $\mathcal{L}$-sentences that are true over $M$.
}.
\item $\o{M}=\RR$, $\o{+}=\min$, $\o{\cdot}=+$, $\o{{}^{-1}}=-$.
\item $\overline{f(\epsilon)}=-\lim\limits_{\epsilon\to +0}\epsilon\log f(\epsilon)$.
\end{itemize}
The word ``$\mathcal{L}$-term'' means a subtraction-free rational function.
The morphism $M\to \o{M};f(\epsilon)\mapsto \o{f(\epsilon)}$ is called the {\it ultradiscretization} or {\it tropicalization}.

The correspondence among combinatorial, tropical, and geometric objects is given as follows:

$$
\vbox{\offinterlineskip%
\halign{\vrule width 0pt height 15pt depth 10pt%
\hfill $\vcenter{\hbox{#}}$ \hfill \vrule&\quad\hfill $\vcenter{\hbox{#}}$\hfill \quad\vrule &\quad\hfill $\vcenter{\hbox{#}}$\hfill\cr
Combinatorial & Tropical ($\simeq$ ultradiscrete) & Geometric \cr \hline\hline
%
%
\vbox{\hbox{Jeu de taquin slide} \hbox{starting from $k^\mathrm{th}$ row}} & The map $\varphi_k$ (\S \ref{sec:2.3}) & \vbox{\hbox{The discrete eq.~(\ref{eq:Laxform})} \hbox{$(L_j^t)_j\mapsto (L_j^{t+1})_j$,} 
\hbox{where $\o{R_0^t}=\o{E}([k])$}
} \cr\hline
 Rectification & Planer diagram (\ref{eq:diagram-of-rectifiction}) 
 & Composition of (\ref{eq:Laxform})
 \cr\hline
%
\vbox{\hbox{Sub-diagram consisted} \hbox{of empty boxes}}
& Associated tableau & Geometric tableau\cr\hline
Shape equivalent class &Associated circled array& \vbox{\hbox{``$F$-matrix version'' of}  \hbox{Geometric tableau \hspace{-15pt}\phantom{$k^\mathrm{th}$} (\ref{eq:F-geometric-tableau})}}\cr
}
}
$$

\section{Tropical (ultradiscrete) KP and jeu de taquin}\label{sec:2}

In this section, we shortly review the previous results~\cite{iwao2018rims}.
For definitions of the terms {\it ``jeu de taquin slide,'' ``inside corner,'' ``outside corner,''etc}, see \S \ref{sec:appA}.

\subsection{Tropical KP equation}\label{sec:2.1}

Let us consider the tropical KP equation:
\begin{equation}\label{eq:udKP}
F_{i,j}^t+F_{i,j+1}^{t+1}=\max\left[
F_{i+1,j+1}^t+F_{i-1,j}^{t+1},
F_{i,j+1}^t+F_{i,j}^{t+1}
\right].
\end{equation}
The following theorem is due to~\cite{kakei2015en}:
\begin{thm}[\cite{mikami2012en,kakei2015en}]
For a sequence of skew tableaux $S^0,S^1,S^2,\dots$, set
\[
F_{i,j}^t=
\left(
\begin{array}{cc}
\mbox{the number of boxes in $1^\mathrm{st},2^\mathrm{nd},\dots,i^{\mathrm{th}}$ rows in $S^t$ }\\
\mbox{which are indexed by a number smaller than or equal to $j$}
\end{array}
\right),
\]
where an empty box is considered as a box indexed by $0$.
If each $S^{t+1}$ is obtained from $S^t$ by a jeu de taquin slide, $(F_{i,j}^t)_{i,j,t}$ satisfies the tropical KP equation $(\ref{eq:udKP})$.
\end{thm}

\begin{prop}[\cite{iwao2018rims}]
By setting
\begin{equation}\label{prop:change}
\begin{aligned}
&Q_{i,j}^t=F_{i,j}^t+F_{i-1,j}^{t+1}-F_{i-1,j}^t-F_{i,j}^{t+1},\\
&W_{i,j}^t=F_{i,j}^t+F_{i,j+1}^{t}-F_{i-1,j}^t-F_{i+1,j+1}^{t},
\end{aligned}
\end{equation}
we have
\begin{equation}\label{eq:QandW}
\begin{aligned}
&Q_{i+1,j+1}^t=(\min[Q_{i+1,j}^t,W_{i+1,j}^t]-\min[Q_{i,j}^t,W_{i,j}^t])+Q_{i,j}^t,\\
&W_{i,j}^{t+1}=(\min[Q_{i+1,j}^t,W_{i+1,j}^t]-\min[Q_{i,j}^t,W_{i,j}^t])+W_{i,j}^t.
\end{aligned}
\end{equation}
\end{prop}

\subsection{Skew tableau and matrix $W$}\label{sec:2.2}

For a skew tableau, let $F_{i,j}$ be the number of $1,2,\dots,j$'s contained in the $1^{\mathrm{st}},2^{\mathrm{nd}},\dots,i^{\mathrm{th}}$ rows.
Define
\begin{align*}
W_{i,j}
&=F_{i,j}+F_{i,j+1}-F_{i-1,j}-F_{i+1,j+1}\\
&=\sharp \{1,2,\dots,j\mbox{'s in the } i^{\mathrm{th}} \mbox{ row}\}-\sharp\{1,2,\dots,(j+1)\mbox{'s in the } (i+1)^{\mathrm{th}} \mbox{ row}\}.
\end{align*}
By definition of skew tableaux, $W_{i,j}$ must be non negative.
Moreover, $W_{i,j}$ satisfies the following conditions:
\begin{gather}
\mbox{There exists some $N$ such that $j>N\Rightarrow W_{i,j}=W_{i,j+1}$ for all $i$}.\label{eq:cond1}\\
\mbox{There exists some $d$ such that $i>d\Rightarrow W_{i,j}=0$ for all $j$}.\label{eq:cond2}\\
\textstyle \sum_{p\geq 0}W_{i+p,j+p}\geq \sum_{p\geq 0}W_{i+1+p,j+p}.\label{eq:cond3}
\end{gather}

Let $\Omega$ be the set of skew tableaux.
Define
\begin{equation*}
\begin{aligned}
\mathfrak{X}:=\{(W_{i,j})_{
\kasane{i\geq 1}{j\geq 0 \hfill}
}\,\vert\,W_{i,j}\in \ZZ_{\geq 0},\ \mbox{with conditions (\ref{eq:cond1}), (\ref{eq:cond2}), (\ref{eq:cond3})}
\}.
\end{aligned}
\end{equation*}
Consider the mapping $W:\Omega\to \mathfrak{X}$ which corresponds a skew tableau with the matrix $(W_{i,j})_{i,j}$.
\begin{prop}[\cite{iwao2018rims}]
$W$ is bijective.
\end{prop}

We always identify $\Omega$ with $\mathfrak{X}$.
While the matrix $W$ is of infinite size, it is ``essentially finite'' because of (\ref{eq:cond1}) and (\ref{eq:cond2}).
To simplify notations, we often regard $W$ as a matrix of size $d\times N$.

\subsection{Jeu de taquin slide starting from the $k^{\mathrm{th}}$ row}\label{sec:2.3}

Now we construct the map $\varphi_k:\mathfrak{X}\to \mathfrak{X}$ for any positive integer $k$, which is a tropical counterpart of the jeu de taquin slide starting from $k^\mathrm{th}$ row.
Let $W=(W_{i,j})\in \mathfrak{X}$ and $[k]=(0,\dots,0,
\stackrel{\stackrel{k}{\vee}}{1},\allowbreak 0,\dots)$.
The definition of $\varphi_k$ is given as follows~\cite{iwao2018rims}:
\begin{enumerate}
\item Let $\bm{Q}_0=(Q_{1,0},Q_{2,0},\dots):=[k]$.
\item
Compute $Q_{i,j}$, $W_{i,j}^+$ recursively as follows:
If the vector $\bm{Q}_j=(Q_{1,j},Q_{2,j},\dots)$ is already defined for some $j\in \ZZ_{\geq 0}$, define the new vectors $\bm{Q}_{j+1}=(Q_{1,j+1},Q_{2,j+1},\dots)$ and $\bm{W}^+_j=(W^+_{1,j},W^+_{2,j},\dots)$ by the formula
\begin{equation}\label{eq:R-matrix}
\begin{cases}
Q_{i+1,j+1}=(\min[Q_{i+1,j},W_{i+1,j}]-\min[Q_{i,j},W_{i,j}])+Q_{i,j},\\
W_{i,j}^+=(\min[Q_{i+1,j},W_{i+1,j}]-\min[Q_{i,j},W_{i,j}])+W_{i,j},
\end{cases}
\end{equation}
where $Q_{0,j}=0$, $W_{0,j}=+\infty$.
(Compare with (\ref{eq:QandW})).
\end{enumerate}
Define $\varphi_k(W)=(W_{i,j}^+)$.

We regard (\ref{eq:R-matrix}) as a recurrence formula whose inputs are $\bm{Q}_j$ and $\bm{W}_j$, and outputs are $\bm{Q}_{j+1}$ and $\bm{W}^+_j$.
It is convenient to draw a diagram such as 
$
\sayou{ & \bm{W}_j & \\
\bm{Q}_j & \+ & \bm{Q}_{j+1}\\
 & \bm{W}^+_j
}$, where the inputs are written on the left and top sides, and the outputs are on the right and bottom sides.
The whole procedure to calculate $\varphi_k(W)$ is diagrammatically displayed as
\begin{equation}\label{eq:one-jeu-de-taquin-slide}
\sayou{
 & \bm{W}_0 & & \bm{W}_1 & & \bm{W}_2 & & \bm{W}_3 & \\
\bm{Q}_0=[k] & \+ & \bm{Q}_1 & \+ & \bm{Q}_2 & \+ & \bm{Q}_3 & \+ &\cdots\\
 & \bm{W}^+_0 & & \bm{W}^+_1 & & \bm{W}^+_2 & & \bm{W}^+_3 & \\
}.
\end{equation}

Moreover, the map $\varphi_k$ also admits a combinatorial interpretation as follows:
\begin{itemize}
\item Draw a path on the matrix $W=(W_{i,j})$ by the following rule (see \S \ref{example:first}): 
\begin{itemize}
\item The path starts at the $(k,0)^{\mathrm{th}}$ position.
\item When the path reaches at the $(i,j)^{\mathrm{th}}$ position, extend it to the lower right neighbor if $W_{i,j}=0$, or to the right neighbor if $W_{i,j}\neq 0$.
\end{itemize}
\item For each non-zero number $W_{i,j} $ on the path, decrease it by one and increase the number at the upper neighbor by one; $W_{i,j}\mapsto W_{i,j}-1$, $W_{i-1,j}\mapsto W_{i-1,j}+1$.
The matrix given by this procedure coincides with $\varphi_k(W)$.
\item The matrix $Q=(Q_{i,j})_{i,j}$ is given by putting $Q_{i,j}=1$ if the path goes through the $(i,j)^\mathrm{th}$ position, and $Q_{i,j}=0$ otherwise.
\end{itemize}

\subsection{Example}\label{example:first}

The jeu de taquin slide
\[
\Tableau{\bl & \bl & \bl & 2 & 3 \\ \gray & 1 & 3 & 4 \\ 2 & 2 & 4 \\ }
\quad
\Tableau{\bl & \bl & \bl & 2 & 3 \\ 1 & \gray & 3 & 4 \\ 2 & 2 & 4 \\ }
\quad
\Tableau{\bl & \bl & \bl & 2 & 3 \\ 1 & 2 & 3 & 4 \\ 2 & \gray & 4 \\ }
\quad
\Tableau{\bl & \bl & \bl & 2 & 3 \\ 1 & 2 & 3 & 4 \\ 2 & 4 & \gray \\ }
\]
corresponds with the matrices
\[
W=
\left(
\haiti{
 1 & 1 & 1 & 1 & 1 & 1 \\
1\r&0\d& 0 & 0 & 1 & 1 \\
 0 & 0 &2\r&2\r&3\r& 3
}
\right),\quad
Q=
\left(
\haiti{
0 & 0 & 0 & 0 & 0 & 0 \\
1 & 1 & 0 & 0 & 0 & 0 \\
0 & 0 & 1 & 1 & 1 & 1
}
\right).
\]
The matrix $\varphi_k(W)$ is given as
$
\left(
\haiti{
2 & 1 & 1 & 1 & 1 & 1 \\
0 & 0 & 1 & 1 & 2 & 2 \\
0 & 0 & 1 & 1 & 2 & 2
}
\right)
$.

\subsection{Rectification}

Any skew tableau reaches a (non-skew) tableau thorough a sequence of finitely many jeu de taquin slides.
To repeat jeu de taquin slides is nothing but to choose inside corners repeatedly.
By filling numbers in chosen inside corners in decreasing order, one obtains a standard tableau.
For example, if we apply the sequence of jeu de taquin slides to
\[
\Tableau{
\bl & \bl & \bl & 2 & 3\\
\bl & 1 & 3 & 4\\
2 & 2 & 4
}\qquad \mbox{defined by}\qquad
\Tableau{
\lgray 1 & \lgray 2 & \lgray 3\\
\lgray 4
},
\]
we obtain the sequence of skew tableaux
\[
\Tableau{
\bl & \bl & \bl & 2 & 3\\
\gray & 1 & 3 & 4\\
2 & 2 & 4
}\ \to\ 
\Tableau{
\bl & \bl & \gray & 2 & 3\\
1 & 2 & 3 & 4\\
2 & 4
}\ \to\ 
\Tableau{
\bl & \gray & 2 & 3\\
1 & 2 & 3 & 4\\
2 & 4
}\ \to\ 
\Tableau{
\gray & 2 & 2 & 3\\
1 & 3 & 4\\
2 & 4
}\ \to\ 
\Tableau{
1 & 2 & 2 & 3\\
2 & 3 & 4 \\
4\\  
}.
\]
We call the (non-skew) tableau obtained by this procedure the {\it rectified tableau}.
With diagrammatic expressions as in \S \ref{sec:2.3}, this procedure is displayed as
\begin{equation}\label{eq:diagram-of-rectifiction}
\sayou{
&(1,1,0)&&(1,0,0)&&(1,0,2)&&(1,0,2)&&(1,1,3)&\cdots \\
[2]&\+&[2]&\+&[3]&\+&[3]&\+&[3]&\+&[3]& \cdots\\
&(2,0,0)&&(1,0,0)&&(1,1,1)&&(1,1,1)&&(1,2,2)&\cdots \\
[1]&\+&[1]&\+&[1]&\+&[1]&\+&[1]&\+&[1]& \cdots\\
&(1,0,0)&&(0,0,0)&&(0,1,1)&&(0,1,1)&&(0,2,2)&\cdots \\
[1]&\+&[1]&\+&[2]&\+&[2]&\+&[2]&\+&[2]& \cdots\\
&(0,0,0)&&(0,0,0)&&(1,0,1)&&(1,0,1)&&(1,1,2)&\cdots \\
[1]&\+&[2]&\+&[3]&\+&[3]&\+&[3]&\+&[3]& \cdots\\
&(0,0,0)&&(0,0,0)&&(1,1,0)&&(1,1,0)&&(1,2,1)&\cdots \\
}.
\end{equation}
If $[k]$ is placed on the leftmost column, it represents the jeu de taquin slide starting from the $k^\mathrm{th}$ row．
The diagram (\ref{eq:diagram-of-rectifiction}) presents a sequence of jeu de taquin slides starting from the $2^\mathrm{nd}$, $1^\mathrm{st}$, $1^\mathrm{st}$, $1^\mathrm{st}$ rows.
The vectors at the bottom correspond with the rectified tableau.
In fact,
\[
\Tableau{
1 & 2 & 2 & 3\\
2 & 3 & 4 \\
4\\  
}\longleftrightarrow
\left(
\haiti{
0 & 0 & 1 & 1 & 1 & 1 \\
0 & 0 & 1 & 1 & 2 & 2 \\
0 & 0 & 0 & 0 & 1 & 1
}
\right).
\]

\subsection{Associated tableaux}\label{sec:2.6}

For diagrams such as (\ref{eq:diagram-of-rectifiction}), we assign the tableau $\Tableau{t_1}\la \Tableau{t_2}\la \dots\la \Tableau{t_d}$ to a column whose entries are $[t_1],[t_2],\dots,[t_d]$ from top to bottom.
For example, from (\ref{eq:diagram-of-rectifiction}), we induce the new diagram
\begin{equation}\label{eq:diagram-of-rectifiction-tate}
\sayou{
&(1,1,0)&&(1,0,0)&&(1,0,2)&&(1,0,2)&&(1,1,3)&\cdots \\
\Tableau{1&1&1\\2}&\+&\Tableau{1&1&2\\2}&\+&\Tableau{1&2&3\\3}&\+&
\Tableau{1&2&3\\3}&\+&\Tableau{1&2&3\\3}&\+&\cdots\\
&(0,0,0)&&(0,0,0)&&(1,1,0)&&(1,1,0)&&(1,2,1)&\cdots \\
}
\end{equation}
It was proved by Iwao~\cite{iwao2018rims} that if one exchanges the entries in the leftmost column in (\ref{eq:diagram-of-rectifiction}) with $[t'_1],[t'_2],\dots,[t'_d]$ that gives the same tableau, one also obtain the same diagram (\ref{eq:diagram-of-rectifiction-tate}).
This means that the diagram (\ref{eq:diagram-of-rectifiction-tate}) is well-defined independently of (\ref{eq:diagram-of-rectifiction}).
The tableau associated with each column is called the {\it associated tableau}.

\begin{defi}
Let $U(\mu)$ be the tableau of shape $\mu$ where any number in $i^\mathrm{th}$ row is $i$.
\end{defi}

\begin{lemma}[Iwao~\protect{\cite[Lemma 5.5]{iwao2018rims}}, (also see Fulton~\protect{\cite[\S 5.2, Lemma 1]{fulton_1996}})]
The associated tableau of any standard tableau of shape $\mu$ is $U(\mu)$.
\end{lemma}
\begin{cor}[The uniqueness of rectification]
The rectified tableau of a skew tableau is unique.
\end{cor}

\subsection{Lift to $\mathcal{M}$}
All procedures in \S \ref{sec:2.1}--\S \ref{sec:2.6} can be expressed in terms of the language $\mathcal{L}$ and the $\mathcal{L}$-structure $\o{\mathcal{M}}$.
They can be lifted to $\mathcal{M}$ appropriately, and their ``geometric'' counterparts are expressed as a discrete integrable system. 

For example, let
\[
R_j^t:=\left(
\begin{array}{ccccc}
I_{1,j}^t & 1 &  &   \\ 
 & I_{2,j}^t & 1 &  \\ 
 &  & I_{3,j}^t & \ddots   \\ 
 &  &  & \ddots 
\end{array} 
\right),\quad
L_j^t:=\left(
\begin{array}{ccccc}
1 &  &  &   \\ 
-V_{1,j}^t & 1 & &  \\ 
 & -V_{2,j}^t & 1 &    \\ 
 &  & -V_{3,j}^t & 1 \\ 
 &  &  & \ddots &\ddots
\end{array} 
\right)^{-1}.
\]
We can uniquely define the rational map
$$
M^\infty\to M^\infty;\quad (I_{i,j}^t,V_{i,j}^t)_{i=1,2,\dots}\mapsto (I_{i,j+1}^t,V_{i,j}^{t+1})_{i=1,2,\dots}
$$
by the equation
\begin{equation}\label{eq:Laxform}
R_j^tL_j^t=L_j^{t+1}R_{j+1}^{t},
\end{equation}
which is a discrete analog of the Toda equation.
It is verified that its tropicalization coincides with (\ref{eq:QandW}), where $Q_{i,j}^t=\o{I_{i,j}^t}$ and $W_{i,j}^t=\o{V_{i,j}^t}$ are the tropical variables.

All other procedures and facts can also be lifted to $\o{\mathcal{M}}$ and expressed in terms of the discrete integrable systems.
See \cite{iwao2018rims} for details.

\section{Application 1: shape equivalence class}\label{sec:3}

\subsection{Definition of shape equivalence}

The rectification of
\[
\Tableau{\bl & \bl & \bl & 2\\
\bl & 1 & 3 & 6 \\
4 & 5 & 7}\qquad
\mbox{defined by}\qquad
\Tableau{\lgray 1 &\lgray  2 &\lgray 3 \\
\lgray 4}
\]
is expressed as
\[
\Tableau{\bl & \bl & \bl & 2\\
\gray & 1 & 3 & 6 \\
4 & 5 & 7}\ \to \ 
\Tableau{\bl & \bl & \gray & 2\\
1 & 3 & 6 \\
4 & 5 & 7}\ \to \ 
\Tableau{\bl & \gray & 2\\
1 & 3 & 6 \\
4 & 5 & 7}\ \to \ 
\Tableau{\gray & 2 & 6\\
1 & 3 & 7 \\
4 & 5 }\ \to \ 
\Tableau{1 & 2 & 6\\
3 & 5 & 7 \\
4 }.
\]
On the other hand, the rectification of
\[
\Tableau{\bl & \bl & \bl & 1\\
\bl & 2 & 2 & 2 \\
3 & 3 & 4}\qquad
\mbox{defined by}\qquad
\Tableau{\lgray 1 &\lgray  2 &\lgray 3 \\
\lgray 4},
\]
is expressed as
\[
\Tableau{\bl & \bl & \bl & 1\\
\gray & 2 & 2 & 2 \\
3 & 3 & 4}\ \to \ 
\Tableau{\bl & \bl & \gray & 1\\
2 & 2 & 2  \\
3 & 3 & 4}\ \to \ 
\Tableau{\bl & \gray & 1\\
2 & 2 & 2  \\
3 & 3 & 4}\ \to \ 
\Tableau{\gray & 1& 2\\
2 & 2 & 4  \\
3 & 3 }\ \to \ 
\Tableau{1 & 2& 2\\
2 & 3 & 4  \\
3  }.
\]
Note that the shapes of these two sequence coincides with each other.
In such case, two skew tableaux are said to {\it have the same shape changes by } 
$
\Tableau{\lgray 1 &\lgray  2 &\lgray 3 \\
\lgray 4}.
$
The following theorem is referred to as the ``shape change theorem~\cite[Appendix A]{fulton_1996}.''
\begin{thm}[Shape change theorem]\label{thm:shape-change}
If two skew tableaux have the same shape changes by some standard tableau, then they actually have the same shape changes by any standard tableau.
\end{thm}

If two skew tableaux have the same shape changes by some standard tableau (therefore, if they have the same shape changes by any standard tableau), they are said to be {\it shape equivalent}.

\begin{example}\label{example:shape-equivalent}
If we apply the sequence of jeu de taquin slides defined by 
$
\Tableau{\lgray 1 &\lgray  3 &\lgray 4 \\
\lgray 2}
$
to the two skew tableaux above, we have
\[
\Tableau{\bl & \bl & \gray & 2\\
\bl & 1 & 3 & 6 \\
4 & 5 & 7}\ \to \ 
\Tableau{\bl & \gray & 2 & 6\\
\bl & 1 & 3  \\
4 & 5 & 7}\ \to \ 
\Tableau{\bl & 1 & 2 & 6\\
\gray & 3 & 7  \\
4 & 5 }\ \to \ 
\Tableau{\gray & 1 & 2 & 6\\
3 & 5 & 7  \\
4  }\ \to \ 
\Tableau{ 1 & 2 & 6\\
3 & 5 & 7  \\
4  }
\]
and 
\[
\Tableau{\bl & \bl & \gray & 1\\
\bl & 2 & 2 & 2 \\
3 & 3 & 4}\ \to \ 
\Tableau{\bl & \gray & 1 & 2\\
\bl & 2 & 2  \\
3 & 3 & 4}\ \to \ 
\Tableau{\bl & 1 & 2 & 2\\
\gray & 2 & 4  \\
3 & 3 }\ \to \ 
\Tableau{\gray & 1 & 2 & 2\\
2 & 3 & 4  \\
3  }\ \to \ 
\Tableau{1 & 2 & 2\\
2 & 3 & 4  \\
3  }.
\]
\end{example}

\subsection{Circled array}

In the sequel, we will give a proof of Theorem \ref{thm:shape-change}.

For real vectors $I=(I_1,I_2,\dots)$ and $V=(V_1,V_2,\dots)$, we define the matrices $E(I)$, $F(V)$ of infinite size as 
\[
E(I)=\left(
\begin{array}{cccc}
I_1 & 1 &  &  \\ 
 & I_2 & 1 &  \\ 
 &  & I_3 & \ddots \\ 
 &  &  & \ddots
\end{array} 
\right),\quad
F(V)=\left(
\begin{array}{cccc}
1 &  &  &  \\ 
-V_1 & 1 &  &  \\ 
 & -V_2 & 1 &  \\ 
 &  & \ddots & \ddots
\end{array} 
\right).
\]
Moreover, for a real infinite vector $V$, define the matrix $F_k(V)$ by
\[
F_k(V)=\left(
\begin{array}{cc}
\mathrm{Id}_{k-1} &  \\ 
 & F(V)
\end{array} 
\right).
\]

Consider the map $M^\infty\to M^\infty;(V_{i,j})\mapsto (U_{i,j})$ uniquely defined by the equation
\begin{equation}\label{eq:F-geometric-tableau}
F(V_N)\cdots F(V_1)F(V_0)=F_{N+1}(U_{N+1})\cdots F_2(U_2)F_1(U_1).
\end{equation}
By induction on $N$, we can verify that this map is expressed by $\mathcal{L}$-terms (\S \ref{sec:1.3}).
\begin{rem}
This map can be regarded as an ``$F$-matrix version'' of Noumi-Yamada's geometric tableau.
In fact, the geometric tableau $(I_{i,j})\mapsto (J_{i,j})$ is defined by the equation 
\[
E(I_\ell)\cdots E(I_2)E(I_1)=E_{1}(J_1)E_2(J_2)\cdots E_\ell(J_\ell),
\]
where $E_k(J)$ is a matrix analogously defined to $F_k(V)$.
For details, see \cite[Section 5]{iwao2018rims}.
\end{rem}

Let $W_{i,j}=\o{V_{i,j}}$ and $L_{i,j}=\o{U_{i,j}}$ be the tropical variables.
Through the tropicalization, we obtain the piecewise linear map $\o{M}^\infty\to \o{M}^\infty;\{W_{i,j}\}\mapsto \{L_{i,j}\}$.
This map has a interesting combinatorial interpretation which we explain below.

A {\it circled array} is a collection of finitely many rows that consists of {\it circled numbers}, where the numbers are arranged in increasing ordering.
(Empty rows are allowed.)
For example, $
\Tableau{
\maru{1} & \maru{2} & \maru{2} & \maru{3}\\
\maru{1} & \maru{1} & \maru{3} \\
\emp\\
\maru{2}
}$ is a circled array.
It consists of $4$ rows, one of which is empty.
We define the action of an {\it one-rowed array} (e.g.~
$
\Tableau{1&1&2&3}
$)
to a circled array, which we will express as
\[
\Tableau{
\maru{1} & \maru{2} & \maru{2} & \maru{3}\\
\maru{1} & \maru{1} & \maru{3} \\
\emp\\
\maru{2}
}\leftharpoonup
\Tableau{1&1&2&3},
\]
by the following manner:
\begin{enumerate}
\item The action is calculated row by row.
Arrange the members of the $1^\mathrm{st}$ row of the circled array and the members of the one-rowed array in the following order: circled $1$'s, boxed $1$'s, circled $2$'s, boxed $2$'s,\dots.
For the example above, we have
\[
\Tableau{\maru{1}&1&1&\maru{2}&\maru{2}&2&\maru{3}&3}.
\]
\item Move the circles according to the time evolution rule of the {\it Takahashi-Satsuma box ball system}~\cite{takahashi1990soliton}.
(See \S \ref{sec:appB}.)
Here a circled number is regarded as a ``box with a ball,'' and a boxed number is regarded as an ``empty box.''
We neglect all the balls that go out from the right end.
\[
\Tableau{ & \maru{}& & & & \maru{}& & \maru{}}
\]
\item Number the circles and boxes according to the following rules: 
(i) Decrease each number by one at where a circle is replaced with a box.
(ii) The numbers at the remaining places do not change.
(iii) Delete boxed $0$'s. 
\[
\Tableau{\hbox to 0pt{\hspace{-5pt}\scalebox{1.6}{$\diagup$}} 0& \maru{1}& 1& 1& 1& \maru{2}& 2& \maru{3}}
\]
\item 
The sequence of circled numbers in the diagram (e.g.~$\Tableau{\maru{1}&\maru{2}&\maru{3}}$) is the $1^\mathrm{st}$ row of the new circled array.
The sequence of boxed numbers (e.g.~$\Tableau{1&1&1&2}$) proceeds to the $2^\mathrm{nd}$ row.
\item Repeat recursively the same procedures to $2^\mathrm{nd},3^\mathrm{rd},\dots$ rows  until the one-rowed array reaches to the bottom of the array.
Finally, the one-rowed array is added to the bottom end as a sequence of circled numbers.
\end{enumerate}
In the exemplified case, the action is calculated as
\[
\Tableau{
\maru{1} & \maru{2} & \maru{2}&\maru{3}\\
\maru{1} & \maru{1} & \maru{3}\\
\emp\\
\maru{2}
}\quad
\raise 18pt\hbox {$\leftharpoonup \Tableau{1&1&2&3}$}
\qquad
\Tableau{
\maru{1} & \maru{2} & \maru{3} \\
\maru{1} & \maru{1} & \maru{3} \\
\emp\\
\maru{2}
}\quad
\raise 6pt\hbox {$\leftharpoonup \Tableau{1&1&1&2}$}
\qquad
\Tableau{
\maru{1} & \maru{2} & \maru{3} \\
\maru{1} & \maru{1} \\
\emp\\
\maru{2}
}
\raise -6pt\hbox {$\leftharpoonup \Tableau{1&2&2}$}
\]
\[
\Tableau{
\maru{1} & \maru{2} & \maru{3} \\
\maru{1} & \maru{1} \\
\emp\\
\maru{2}
}
\raise -18pt\hbox {$\leftharpoonup \Tableau{1&2&2}$}
\qquad\qquad
\Tableau{
\maru{1} & \maru{2} & \maru{3} \\
\maru{1} & \maru{1} \\
\emp\\
\maru{2}
}
\raise -30pt\hbox {$\leftharpoonup \Tableau{1&1&2}$}
\qquad
\Tableau{
\maru{1} & \maru{2} & \maru{3} \\
\maru{1} & \maru{1} \\
\emp\\
\maru{2}\\
\maru{1}&\maru{1}&\maru{2}}
\]
Let
\[
x_1 \leftharpoonup x_2 \leftharpoonup x_3\leftharpoonup\cdots.
\]
denote the circled array given by acting the one-rowed arrays $x_1,x_2,x_3,\dots$ to the empty array.

\begin{thm}\label{thm:geometric-circled-tableau}
Let $\bm{W}_j=(W_{1,j},W_{2,j},\dots)$ be a sequence of nonnegative integers that satisfies $\sum_{i}W_{i,j}<\infty$.
Let $w_j$ denote the one-rowed array that consists of $W_{1,j}$ $1$'s, $W_{2,j}$ $2$'s,\dots.
Then the map $(W_{i,j})\mapsto (L_{i,j})$ has the following combinatorial interpretation:
The number of $i$'s in the $j^\mathrm{th}$ row of the circled array
\[
w_N \leftharpoonup w_{N-1} \leftharpoonup \cdots \leftharpoonup w_0
\]
equals to $L_{i,j}$.
Here $N$ is a sufficiently large integer.
\end{thm}
\begin{proof}
Let $V,V',U,U'\in M^\infty$ be vectors of infinite length. 
For any $k$, there uniquely exists the map $M^\infty\to M^\infty;(U,V)\mapsto (U',V')$ defined by
\[
F_k(U)F_k(V)=F_{k+1}(V')F_k(U'),
\]
which is obviously equivalent to $F_1(U)F_1(V)=F_{2}(V')F_1(U')$.
This equation is noting but the Lax formulation of the discrete Toda equation (see, for example, ).
It is well known that its tropical counterpart is the Takahashi-Satsuma box-ball system.
In this context, the action $w_a\leftharpoonup w_b$ can be seen as the tropical counterpart of the equation $F(V_a)F(V_b)=F_{2}(U^{(2)})F_1(U^{(1)})$, where $V_a,V_b$ associate with $w_a,w_b$ and $U^{(i)}$ associated with the $i^\mathrm{th}$ row of the circled array.
Generally, the action $(w_{a_1}\leftharpoonup w_{a_2}\leftharpoonup\cdots w_{a_j})\leftharpoonup w_b$ is the counterpart of
\begin{align*}
&F_{j}(U^{(j)})\cdots F_3(U^{(3)})F_2(U^{(2)})\underline{F_{1}(U^{(1)})F(V_b)}\\
&=F_{j}(U^{(j)})\cdots F_3(U^{(3)})\underline{F_2(U^{(2)})F_2(V_{b_1})}F_1(U^{(1)'})\\
&=F_{j}(U^{(j)})\cdots \underline{F_3(U^{(3)})F_3(V_{b_2})}F_2(U^{(2)'})F_1(U^{(1)'})\\
&=\cdots
=F_{j+1}(U^{(j+1)'})\cdots F_3(U^{(3)'})F_2(U^{(2)'})F_1(U^{(1)'}),
\end{align*}
where the actions $F_k(U^{(k)})F_k(V_{b_{k-1}})=F_{k+1}(V_{b_k})F_k(U^{(k)'})$ are underlined.
This proves the relation between (\ref{eq:F-geometric-tableau}) and the circled array.
\end{proof}


\subsection{Associated circled array}

Let $N$ be a sufficiently large number.
It is verified that one can uniquely define the map $M^\infty\to M^\infty;(V_i,I)_{i=0,1,\dots}\mapsto (V'_i,I')_{i=0,1,\dots}$ by the equation 
\begin{equation*}
F(V_N')\cdots F(V_1')F(V_0')E(I)=E(I')F(V_N)\cdots F(V_1)F(V_0).
\end{equation*}
It is also verified that this map can be expressed by $\mathcal{L}$-terms.
The tropicalization $\o{M}^\infty \to \o{M}^\infty;(W_i,Q)_{i=0,1,\dots}\mapsto (W'_i,Q')_{i=0,1,\dots}$ corresponds with the diagram
\[
\sayou{
 & \bm{W}_0 & & \bm{W}_1 & & \bm{W}_2 & & \bm{W}_N & \\
(\bm{Q}=)\bm{Q}_0 & \+ & \bm{Q}_1 & \+ & \bm{Q}_2 & \+ & \cdots & \+ &\bm{Q}_{N+1}(=\bm{Q}')\\
 & \bm{W}'_0 & & \bm{W}'_1 & & \bm{W}'_2 & & \bm{W}'_N & \\
}.
\]

Let us introduce the new variables $U_{i,j},U_{i,j}'$ by
\begin{gather*}
F(V_N)\cdots F(V_0)=F_{N+1}(U_{N+1})\cdots F_1(U_1),\\
F(V_N')\cdots F(V_0')=F_{N+1}(U_{N+1}')\cdots F_1(U_1').
\end{gather*}
Then we obtain
\begin{gather}\label{eq:def-of-UI}
F_{N+1}(U_{N+1}')\cdots F_2(U_2')F_1(U_1')E(I)=E(I')F_{N+1}(U_{N+1})\cdots F_2(U_2)F_1(U_1).
\end{gather}
\begin{prop}\label{prop:map-of-UI}
There exists the map $M^\infty\to M^\infty;(U_i,I)\mapsto (U_i',I')$ that is uniquely defined by (\ref{eq:def-of-UI}).
Moreover, this map is expressed by $\mathcal{L}$-terms.
\end{prop}
\begin{proof}
It is enough to prove the existence and uniqueness of the map $M^\infty\to M^\infty;(U,I)\mapsto (U',I')$ such that $F_k(U')E(I)=E(I')F_k(U)$ for any $k$, which is easily verified by straightforward calculation.
Obviously, this map is expressed by $\mathcal{L}$-terms.
\end{proof}

From Proposition \ref{prop:map-of-UI}, the map $(U_i,I)\mapsto (U_i',I')$ induces a tropical map $(L_i,Q)\mapsto (L'_i,Q')$, where $L_{i,j}=\o{U_{i,j}}$ and $Q_{i,j}=\o{I_{i,j}}$.
We write this map diagrammatically as
\[
\sayou{ & (\bm{L}_1,\dots,\bm{L}_{N+1}) & \\
\bm{Q}& \+ & \bm{Q}'\\
& (\bm{L}'_1,\dots,\bm{L}'_{N+1})&
}.
\]
From Theorem \ref{thm:geometric-circled-tableau}, the data $(\bm{L}_1,\dots,\bm{L}_{N+1})$ naturally corresponds with a circled array, which we will call the {\it associated circled array}.
Because the data $(\bm{L}_1,\dots,\bm{L}_{N+1})$ is determined by a skew tableau ($\simeq$ an array $W$), there exists a natural correspondence from the set of skew tableaux to the set of circled arrays.
\begin{example}
The skew tableau
\[
\Tableau{
\bl & \bl & \bl & 1 \\
\bl & 2 & 2 & 2\\
3 & 3 & 4
},\quad
W=\left(
\haiti{
2&0&0&0&0\\
1&1&2&1&1\\
0&0&0&2&3
}
\right)
\]
corresponds with the circled array $(N=4)$
\[
\Tableau{2&3&3&3}\leftharpoonup
\Tableau{2&3&3}\leftharpoonup
\Tableau{2&2}\leftharpoonup
\Tableau{2}\leftharpoonup
\Tableau{1&1&2}
=\Tableau{\maru{2}\\\maru{1}\\\maru{1}&\maru{1}\\\emp\\\emp}
\left(
=\Tableau{\maru{2}\\\maru{1}\\\maru{1}&\maru{1}}
\right)
.
\]
To simplify notations, we often omit to write empty rows at the bottom of an array.
We let the sign ``$\bm{\emptyset}$'' denote the array that consists of empty rows.
\end{example}

For example, the rectification of
\[
\Tableau{
\bl & \bl & \bl & 1 \\
\bl & 2 & 2 & 2\\
3 & 3 & 4
}\qquad\mbox{defined by}\qquad
\Tableau{
\lgray 1 & \lgray 2 & \lgray 3\\
\lgray 4
}
\]
is expressed by the following diagram with circled arrays:
\begin{equation}\label{eq:diagram-of-rectification-yoko}
\sayou{
&\Tableau{\maru{2}\\\maru{1}\\\maru{1}&\maru{1}}&\\
[2]&\+&[2]\\
&\Tableau{\maru{1}\\\emp\\\maru{1}&\maru{1}}&\\
[1]&\+&[1]\\
&\Tableau{\emp\\\emp\\\maru{1}&\maru{1}}&\\
[1]&\+&[3]\\
&\Tableau{\emp\\\emp\\\maru{1}}&\\
[1]&\+&[3]\\
&\Tableau{\bl\bm{\emptyset}}&
}.
\end{equation}

\subsection{Proof of Shape change theorem}

We often look the diagrams ``upside-down.''
In other words, we look the data at right and bottom as inputs and at top and left as outputs.
Combinatorially, this corresponds with the {\it reverse slide}~\cite[\S 1.2]{fulton_1996}.
As its name suggests, the reverse slide is the reverse operation of the jeu de taquin slide.


\begin{lemma}\label{lemma:emptyrow}
The associated circled array of any (non-skew) tableau is $\bm{\emptyset}$.
\end{lemma}
\begin{proof}
Let $W=(W_{i,j})$ be the matrix associated with a (non-skew) tableau, and $w_j$ be the one-rowed array associated with the vector $\bm{W}_j=(W_{1,j},W_{2,j},\dots)$.
Since any tableaux contains no empty box, the relation $i>j\Rightarrow W_{i,j}=0$ holds.
Especially, the array $w_j$ consists of numbers lower than or equal to $j$.
For sufficiently large $N$, one can verify by backward induction on $j\leq N$ that all numbers in the circled array
\[
w_N\leftharpoonup w_{N-1}\leftharpoonup\dots \leftharpoonup w_j
\]
are lower than or equal to $j$.
Hence the circled array $w_N\leftharpoonup w_{N-1}\leftharpoonup\dots \leftharpoonup w_0$ consists of empty rows.
\end{proof}

\begin{prop}
Assume that in some tableau, new empty boxes arise in $b_k^\mathrm{th},b_{k-1}^\mathrm{th},\allowbreak\dots,b_1^\mathrm{th}$ rows when a sequence of reverse slides starting from  $a_k^{\mathrm{th}},a_{k-1}^{\mathrm{th}},\dots,a_1^{\mathrm{th}}$ rows are operated.
Then the sequence $b_1,b_2,\dots,b_k$ depends only on the sequence $a_1,a_2,\allowbreak\dots,a_k$, and is independent of the choice of tableaux.
\end{prop}
\begin{proof}
The situation is expressed by a diagram such as (\ref{eq:diagram-of-rectification-yoko}) where the entries of the left column are $[b_1],\dots,[b_k]$ and the entries of the right column are $[a_1],\dots,[a_k]$.
When we see the diagram ``upside-down,'' the sequence $a_1,\dots,a_k$ and $\bm{\emptyset}$ (see Lemma \ref{lemma:emptyrow}) are thought of the inputs.
This implies that $b_1,\dots,b_k$ depends only on $a_1,\dots,a_k$.
\end{proof}

\begin{example}
When a sequence of reverse slides starting from $(2,1,2,4,3)^{\mathrm{th}}$ rows operated to any tableau, empty boxes must arise in $(1,1,2,3,2)^\mathrm{th}$ rows.
For example, see
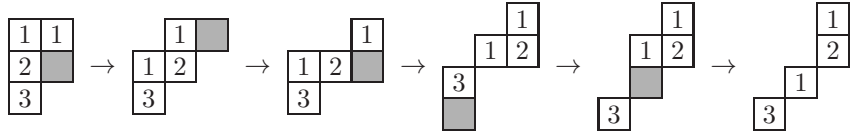
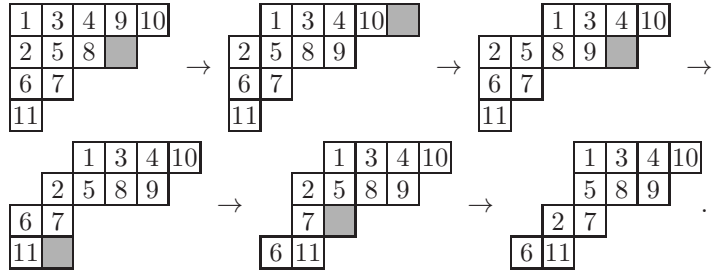
\[
\Tableau{
1&1\\
2&\gray\\
3}\ \to\ 
\Tableau{
\bl&1&\gray\\
1&2\\
3}\ \to\ 
\Tableau{
\bl&\bl&1\\
1&2&\gray\\
3}\ \to\ 
\Tableau{
\bl&\bl&1\\
\bl&1&2\\
3\\
\gray}\ \to\ 
\Tableau{
\bl&\bl&1\\
\bl&1&2\\
\bl&\gray\\
3}\ \to\ 
\Tableau{
\bl&\bl&1\\
\bl&\bl&2\\
\bl&1\\
3}
\]
and
\begin{align*}
&
\Tableau{
1&3&4&9&10\\
2&5&8&\gray\\
6&7\\
11
}\ \to\ 
\Tableau{
\bl&1&3&4&10&\gray\\
2&5&8&9\\
6&7\\
11
}\ \to\ 
\Tableau{
\bl&\bl&1&3&4&10\\
2&5&8&9&\gray\\
6&7\\
11
}\ \to\ \\
&\Tableau{
\bl&\bl&1&3&4&10\\
\bl&2&5&8&9\\
6&7\\
11&\gray
}\ \to\ 
\Tableau{
\bl&\bl&1&3&4&10\\
\bl&2&5&8&9\\
\bl&7&\gray\\
6&11
}\ \to\ 
\Tableau{
\bl&\bl&1&3&4&10\\
\bl&\bl&5&8&9\\
\bl&2&7\\
6&11
}.
\end{align*}
\end{example}

\begin{cor}\label{cor:cir}
The associated circled array of a skew tableau is determined by the sequence $a_1,a_2,\dots,a_k$.
\end{cor}
\begin{proof}
It suffices to see the diagram (\ref{eq:diagram-of-rectification-yoko}).
\end{proof}

Then we have the following theorem which contains the shape change theorem:
\begin{thm}\label{thm:main1}
For two skew tableaux of the same shape, the following $($a--c$)$ are equivalent: 
\begin{enumerate}
\def\theenumi{\alph{enumi}}
\item They are shape equivalent.
\item They have the same shape changes by some standard tableau.
\item Their associated circled arrays coincide with each other.
\end{enumerate}
\end{thm}
\begin{proof}
(a)$\Rightarrow$(b) is obvious.
For (b)$\Rightarrow$(c), assume that two skew tableaux admit a sequence of jeu de taquin slides starting from  $b_1^{\mathrm{th}},b_2^{\mathrm{th}},\dots,b_k^{\mathrm{th}}$ rows, and reach to tableaux where outside corners in $a_1^{\mathrm{th}},a_2^{\mathrm{th}},\dots,a_k^{\mathrm{th}}$ rows has been changed.
Since their associated circled arrays depend on $a_i$'s only (Corollary \ref{cor:cir}), they coincide with each other.
For (c)$\Rightarrow$(a), consider the diagram such as (\ref{eq:diagram-of-rectification-yoko}).
Since the circled arrays at the top are same, the changes that occur outside corners by the sequence of jeu de taquin also coincide.
\end{proof}

From Theorem \ref{thm:main1}, it can be said that the shape equivalent class of a skew tableau is completely determined by its associated circled array.

\begin{example}
\def\v{\vdots}
Both associated circled arrays of the skew tableaux
$
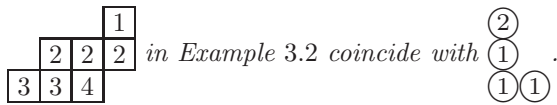
\Tableau{\bl & \bl & \bl & 2\\
\bl & 1 & 3 & 6 \\
4 & 5 & 7}
$
and 
$
\Tableau{\bl & \bl & \bl & 1\\
\bl & 2 & 2 & 2 \\
3 & 3 & 4}
$
in Example $\ref{example:shape-equivalent}$ coincide with
$\Tableau{\maru{2}\\\maru{1}\\\maru{1}&\maru{1}}$.
\end{example}

\begin{rem}
There exists an alternative way to determine the shape equivalent class by calculating the ``$Q$-tableau'' ~\cite[\S AppendixA]{fulton_1996}.
This can be calculated easily, but essentially depends on the shape of the skew tableau.
The circled array, however, depends only on the order of reverse slides (Corollary \ref{cor:cir}).
It would be said that the circled array presents ``universal'' information of the shape equivalent class.
\end{rem}

\section{Application 2: Littlewood-Richardson correspondence}\label{sec:4}

As we have seen in the previous sections, the diagram of rectification (\ref{eq:diagram-of-rectifiction}) induces a new useful diagram if one gets its rows or columns together in one bundle.
If we do both procedures, we obtain the following diagram:
\[
\sayou{&(\bm{L}_1,\dots,\bm{L}_{N+1})&\\
(\bm{P}_1,\dots,\bm{P}_k)&\+&(\bm{P}'_1,\dots,\bm{P}'_k)\\
&(\bm{L}'_1,\dots,\bm{L}'_{N+1})&
}.
\]
This diagram defines the one-to-one correspondence $(L_i,P_i)\leftrightarrow (L_i',P_i')$.
For example, from the diagram (\ref{eq:diagram-of-rectification-yoko}), we have
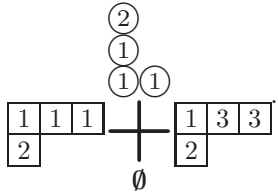
\[
\sayou{
&\Tableau{\maru{2}\\\maru{1}\\\maru{1}&\maru{1}}&\\
\Tableau{1&1&1\\2}&\+&\Tableau{1&3&3\\2}\\
&\bm{\emptyset}&
}.
\]

In general, for any two Young diagrams $\mu\subset \lambda$, we have
\begin{equation}\label{eq:diagram-of-rectification-compact}
\sayou{
&M& \\
U(\mu)&\+&Z\\
&\bm{\emptyset}&
},
\end{equation}
where $M$ is a circled array and $Z$ is a tableau of shape $\mu$.
We will see that the correspondence is related closely with the {\it Littlewood-Richardson correspondence}~\cite[\S A.1]{fulton_1996}.

\subsection{Definition of the Littlewood-Richardson correspondence}\label{sec:4.1}

We review the definition of the Littlewood-Richardson correspondence.
For the definitions of the terms {\it RSK correspondence} and {\it $P$ tableau}, see the standard textbook~\cite{fulton_1996}.

Similar to the case of jeu de taquin slides, repeating reverse slides is nothing but choosing outside corners repeatedly.
In other words, this is equivalent to specify a standard skew tableau ``sticking'' outside of the tableau.
For example, the sequence of reverse slides to
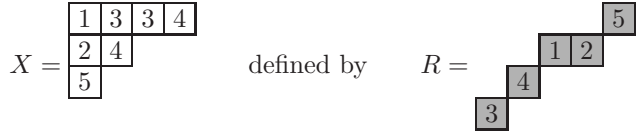
\[
X=
\Tableau{
1 & 3 & 3 & 4\\
2 & 4 \\
5\\
\bl
}
\qquad\mbox{defined by}\qquad
R=
\Tableau{
\bl & \bl & \bl &\bl & \gray 5\\
\bl & \bl & \gray 1 & \gray 2  \\
\bl & \gray 4\\
\gray 3
}
\]
results the skew tableau
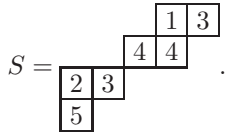
\[
S=
\Tableau{
\bl&\bl&\bl&1&3\\
\bl&\bl&4&4\\
2&3\\
5
}.
\]
There exists a useful correspondence between sequences of reverse slides and {\it tableaux pairs}.
Here we give its outline briefly according to Fulton \cite[\S Appendix A]{fulton_1996}.

Let $R$ be the standard skew tableau sticking outside of the tableau $X$.
For arbitrary tableau $V_\circ$ of shape $\lambda$, one can find a word $w=t_1t_2\dots t_m$ ($m=\vert \mu \vert$) and a tableau $U$ of shape $\mu$ such that
\begin{itemize}
\item $t_1\to t_2\to\cdots \to t_m\to U=V_\circ$,
\item when $t_{m-i}$ is column-bumped, the new box evokes at the place of $\Tableau{\gray i}$
\end{itemize}
by ``reverse column bumping algorithm.'' 
Let $T=P(w)$ denote the $P$-tableau associated with $w$.
We call the pair $[T,U]$ the {\it Littlewood-Richardson pair}.

Let $\mathcal{S}(\lambda/\mu,X)$ be the set of skew tableaux of shape $\lambda/\mu$, the rectification of which is $X$.
On the other hand, let $\mathcal{T}(\lambda,\mu,V_\circ)$ denote the set of the pairs $[T,U]$ where $T$ is of shape $\lambda$, $U$ is of shape is $\mu$, and $T\cdot U=V_\circ$.
Let 
\[
\mathcal{S}(\lambda/\mu,X)\to \mathcal{T}(\lambda,\mu,V_\circ);\qquad S\mapsto [T,U]
\]
be the map that associates the skew tableau $S$, given from $X$ through a sequence of reverse slides defined by $R$, with the Littlewood-Richardson pair $[T,U]$.
It is known that this map is bijective \cite[\S 5.1, Proposition 2]{fulton_1996}.
This map is called the {\it Littlewood-Richardson rule}.
Moreover, for arbitrary two tableaux $X,Y$ of the same shape, one can define the bijection
\[
\mathcal{S}(\lambda/\mu,X)\to \mathcal{T}(\lambda,\mu,V_\circ)\to \mathcal{S}(\lambda/\mu,Y).
\]
This bijection $\mathcal{S}(\lambda/\mu,X)\leftrightarrow\mathcal{S}(\lambda/\mu,Y)$ is called the {\it Littlewood-Richardson correspondence by $V_\circ$}.
Later we will see that it is independent of the choice of $V_\circ$.

\begin{example}
Let us take $X,R,S$ as in the beginning of \S \ref{sec:4.1}.
If one put
$V_\circ=
\Tableau{
1&1&2&2&2\\
2&3&4&4\\
3&4\\
4}
$, the Littlewood-Richardson pair associated with $S$ is
\[
[T,U]
=\left[
\Tableau{
1&3&3\\
4&4},
\Tableau{
1&2&2&2\\
2&4\\
4}
\right].
\]
\end{example}

\subsection{Tropical interpretation of Littlewood-Richardson correspondence}

As we have seen in the previous subsection, the Littlewood-Richardson correspondence is essentially defined via the standard skew tableau $R$.
Let $\mathcal{R}(\lambda,X)$ be the set of standard skew tableaux $R$ sticking outside of $X$, where the outside of $R$ is of shape $\lambda$.
Therefore, we have the commutative diagram:
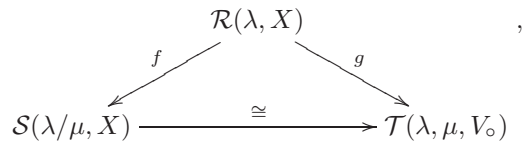
\[
\xymatrix{
&\mathcal{R}(\lambda,X)\ar_f[dl]\ar^g[dr]&\\
\mathcal{S}(\lambda/\mu,X)\ar^{\cong}[rr]& &\mathcal{T}(\lambda,\mu,V_\circ)
},
\]
where $f$ is the reverse column bumping and $g$ is a composition of $f$ and the Littlewood-Richardson rule.

Note that the equation $f(R)=S$ is equivalent to the diagram 
$$
\sayou{
&\bm{W}_0& &\bm{W}_1& &\bm{W}_2& \cdots & \bm{W}_N&  \\
U(\mu)&\+& U_1 &\+& U_2&\+& \cdots &\+& Z=(\mbox{the associated tableau of $R$})\\
&\bm{W}'_0& &\bm{W}'_1& &\bm{W}'_2& \cdots & \bm{W}'_N&
},
$$ 
where the matrices $W=(\bm{W}_j)$ and $W'=(\bm{W}'_j)$ correspond with $S$ and $X$ respectively. 
By introducing the equivalence relation 
\[
R_1\sim R_2\iff \mbox{the associated tableaux of $R_1$ and $R_2$ are same}
\]
over $\mathcal{R}(\lambda,X)$, we induce the two one-to-one correspondences
\[
(\mathcal{R}(\lambda,X)/\sim) \leftrightarrow \mathcal{S}(\lambda/\mu,X),\qquad
(\mathcal{R}(\lambda,X)/\sim) \leftrightarrow \mathcal{T}(\lambda,\mu,V_\circ).
\]
Because the shapes of $U(\mu)$ and $Z$ are same, each element of $\mathcal{R}(\lambda,X)/\sim$ can be identified with its associated tableau of shape $\mu$.
If $\mathcal{T}(\mu)$ denotes the set of tableaux of shape $\mu$, we finally obtain the maps
\[
F:\mathcal{S}(\lambda/\mu,X)\hookrightarrow \mathcal{T}(\mu),\qquad
G:\mathcal{T}(\lambda,\mu,V_\circ)\hookrightarrow \mathcal{T}(\mu),
\]
where 
\[
F(S)=G([T,U])\iff \mbox{$[T,U]$ is the Littlewood-Richardson pair associated with $S$.}
\]

From this relation, we obtain a new characterization of the Littlewood-Richardson correspondence.
\begin{thm}\label{thm:main2}
Let $X,Y$ be two tableau of same shapes.
The two skew tableaux $S_1\in \mathcal{S}(\lambda/\mu,X)$ and $S_2\in \mathcal{S}(\lambda/\mu,Y)$ correspond with each other if and only if $F(S_1)=F(S_2)$.
Especially, the Littlewood-Richardson correspondence does not depend on the choice of $V_\circ$.
\end{thm}
\begin{cor}
Two tableaux of the same shapes are shape equivalent if and only if they correspond with each other through the Littlewood-Richardson correspondence.
\end{cor}
\begin{proof}
Let $S_1,S_2$ be two tableaux of the same shapes and $M_1,M_2$ be their associated circled arrays.
From the diagram 
$
\sayou{
&M_i& \\
U(\mu)&\+& F(S_i)\\
&\bm{\emptyset}&
}
$
($i=1,2$), we have $F(S_1)=F(S_2)\iff M_1=M_2 \iff $ $S_1$ and $S_2$ are shape equivalent.
\end{proof}

\subsection*{Acknowledgment}
This work is partially supported by JSPS KAKENHI:26800062.

\appendix

\section{Basics on the combinatorics of Young tableaux}\label{sec:appA}

A box $B$ in a Young diagram is said to be {\it placed in the corner} if there exists no box below nor on the right of $B$.
For a skew diagram $\lambda/\mu$, the corner of $\lambda$ is called the {\it outside corner} and the corner of $\mu$ is called the {\it inside corner}.

For a skew tableau $T$ and an inside corner $b$, the {\it jeu de taquin starting from $b$} is defined as follows:
(i) Compare the two numbers in the boxes below and on the right of the hole $b$, and slide a box with smaller number to $b$.
If these two numbers are same, slide the box below $b$.
(ii) Compare the two numbers in the boxes below and on the right of the hole, and slide a box according to the same rule in (i).
(iii) Repeat (ii) until the hole reaches to the outside corner.

The following is an example of jeu de taquin.
The grayed boxes denote the holes.
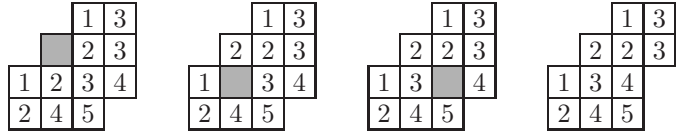
\[
\Tableau{ \bl& \bl& 1 & 3\\ \bl& \gray & 2 & 3\\ 1 & 2 & 3& 4\\ 2&4&5}
\qquad
\Tableau{ \bl& \bl& 1 & 3\\ \bl& 2 & 2 & 3\\ 1 & \gray  & 3&4\\ 2&4&5}
\qquad 
\Tableau{ \bl& \bl& 1 & 3\\ \bl& 2 & 2 & 3\\ 1 & 3  & \gray &4\\ 2&4&5}
\qquad 
\Tableau{ \bl& \bl& 1 & 3\\ \bl& 2 & 2 & 3\\ 1 & 3  & 4\\ 2&4&5}
\]
In this example, a jeu de taquin starts in the $2^\mathrm{nd}$ row, and ends in the $3^\mathrm{rd}$ row.

Let $T$ be a tableau and $t$ be a number.
The {\it row bumping} (or {\it row insertion}) of $t$ to $T$ is defined as follows:
(i) If $t$ is equal to or greater than all the numbers in the $1^\mathrm{st}$ row of $T$, put a new box with $t$ at the place next to the rightmost box of this row.
If not, $t$ ``bumps'' the leftmost number greater than $t$, then the ``bumped'' number goes to the next row.
(ii) Apply similar procedure as (i) to the next row of $T$ and the bumped number.
(iii) Repeat (ii) until no more number is bumped.

The following diagram presents the row bumping of $3$ to the tableau.
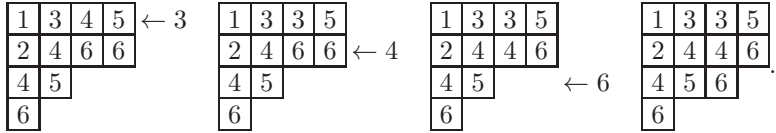
\[
\Tableau{1&3&4&5&\bl \quad\la 3\\2&4&6&6\\4&5\\6}
\qquad
\Tableau{1&3&3&5\\2&4&6&6&\bl \quad\la 4\\4&5\\6}
\qquad
\Tableau{1&3&3&5\\2&4&4&6\\4&5&\bl& \bl &\bl \quad\la 6\\6}
\qquad
\Tableau{1&3&3&5\\2&4&4&6\\4&5&6\\6}.
\]
The tableau obtained by row bumping of $t$ to $T$ is denoted by 
\[
T\la t\qquad \mbox{or}\qquad T\la \Tableau{t}.
\]

Moreover, the {\it column bumping} (or {\it column insertion}) of $t$ to $T$ is defined as follows:
(i) If $t$ is {\it greater than} all the numbers in the $1^\mathrm{st}$ column of $T$, put a new box with $t$ at the place next to the box at the bottom of this column.
If not, $t$ ``bumps'' the number at the highest position among the numbers {\it equal to or greater than} $t$, then the ``bumped'' number goes to the next column.
(ii) Apply similar procedure as (i) to the next column of $T$ and the bumped number.
(iii) Repeat (ii) until no more number is bumped.
The tableau obtained by column bumping of $t$ to $T$ is denoted by 
\[
t\to T\qquad \mbox{or}\qquad \Tableau{t}\to T.
\]

\section{Takahashi-Satsuma Box-Ball system }\label{sec:appB}

Let us consider the map $(U_i,V_i)_{i=1,2,\dots}\mapsto (U'_i,V'_i)_{i=1,2,\dots}$ that is defined by the equation $F_1(U)F_1(V)=F_{2}(V')F_1(U')$.
This map is expressed by $\mathcal{L}$-terms.
Let $L=\o{U_i}$ and $W=\o{V_i}$ be the tropical variable.
Then the tropical map $(L_i,W_i)_{i=1,2,\dots}\mapsto (L'_i,W'_i)_{i=1,2,\dots}$ is expressed as 
\[
L_i'=\min\left[W_i,L_i+\max_k[0,\textstyle \sum_{j=1}^k(L_{i-j}-W_{i-j})]\right],\qquad
W_i'=L_i+W_i-L_i'.
\]
The {\it Takahashi-Satsuma Box-Ball system}~\cite{takahashi1990soliton} is a combinatorial interpretation of this map, which is expressed as follows.

Let us consider a sequence of infinitely many boxes such as 
\begin{equation}\label{eq:BoxBall}
\cdots \Tableau{&&&\tama{}& & \tama{}&\tama{} &\tama{} & & &\tama{}& & & & }\cdots
\end{equation}
Here each box either contains a ball or is empty.
A string of balls is called a {\it soliton}.
For instance, this picture contains three solitons.

Let $L_i$ be the number of balls that are contained in the $i^\mathrm{th}$ soliton from left, and $W_i$ be the distance between $i^\mathrm{th}$ and $(i+1)^\mathrm{th}$ solitons.
Then (\ref{eq:BoxBall}) is expressed as the following combinatorial rule: (i) Create a copy of each ball. 
(ii) Move each copy to the nearest empty box on the right.
(iii) Delete the original balls.
This procedure is called the {\it time evolution role} of the Takahashi-Satsuma Box-Ball system.
The following is an example of the combinatorial expression of the time evolution $(L_i,W_i)_{i=1,2,3}\mapsto (L'_i,W'_i)_{i=1,2,3}$:
\begin{gather*}
(L_i,W_i);\qquad \cdots \Tableau{&&&\tama{}& & \tama{}&\tama{} &\tama{} & & &\tama{}& & & & }\cdots\\
(L'_i,W'_i);\qquad \cdots \Tableau{&&& & \tama{} & & & &\tama{} &\tama{} & &\tama{}&\tama{} &\tama{}  & }\cdots
\end{gather*}
For more detail, see 

\section{The correspondence $Z\leftrightarrow M$}\label{sec:appC}

Here is a short list of examples of the correspondence $Z\leftrightarrow M$:
\[
\Tableau{1}\leftrightarrow\Tableau{\maru{1}}\qquad
\Tableau{1&1}\leftrightarrow\Tableau{\maru{1}&\maru{1}}\qquad
\Tableau{1&1&\cdots\bl&1}\leftrightarrow\Tableau{\maru{1}&\maru{1}&\cdots\bl&\maru{1}}
\]
\[
\Tableau{1&1&2\\2}\leftrightarrow\Tableau{\maru{1}&\maru{2}\\\maru{1}&\maru{1}}
\quad\Tableau{1&1&2\\3}\leftrightarrow\Tableau{\maru{1}&\maru{1}\\\maru{2}\\\maru{1}}
\quad\Tableau{1&2&2\\2}\leftrightarrow\Tableau{\maru{2}\\\maru{1}&\maru{1}&\maru{1}}
\quad\Tableau{1&2&2\\3}\leftrightarrow\Tableau{\maru{1}\\\maru{1}&\maru{2}\\\maru{1}}
\]
\[
\Tableau{1&2\\3}\leftrightarrow\Tableau{\maru{1}\\\maru{2}\\\maru{1}}\qquad
\Tableau{2&3\\4}\leftrightarrow\Tableau{\emp\\\maru{1}\\\maru{2}\\\maru{1}}\qquad
\Tableau{1&3\\2&5\\4}\leftrightarrow\Tableau{\maru{3}\\\maru{2}\\\maru{2}\\\maru{1}\\\maru{1}}\qquad
\Tableau{1&2&2&3\\2&3&3\\4}\leftrightarrow\Tableau{\maru{2}\\\maru{1}&\maru{2}&\maru{3}\\\maru{1}&\maru{1}&\maru{2}\\\maru{1}}
\]
Let $Z^\ast$ denote the {\it Sch\"{u}tzenberger dual tableau}\footnote{See, for example, \cite[Appendix A]{fulton_1996}.} of $Z$, and $i^\ast$ denote the {\it opposite number} of $i$.
We have the following conjecture about the correspondence $Z\leftrightarrow M$:
\begin{conj}
Then the number of $j$'s in the $i^\mathrm{th}$ row of $M$ equals to the number of $i^\ast$'s in the $j^\mathrm{th}$ row of $R^\ast$.
\end{conj}


\end{document}